\documentclass[a4paper]{amsart}

\usepackage[v2,all]{xy}
\usepackage{amsmath,amsfonts,amssymb,graphicx}

\newtheorem{prop}{Proposition}

\theoremstyle{definition}
\newtheorem{rem}{Remark}

\DeclareMathOperator{\HH}{H}
\DeclareMathOperator{\QH}{QH}
\DeclareMathOperator{\UU}{U}

\DeclareMathOperator{\SO}{SO}
\DeclareMathOperator{\Spin}{Spin}
\DeclareMathOperator{\Sp}{Sp}
\DeclareMathOperator{\Gr}{Gr}
\DeclareMathOperator{\OGr}{OGr}

\DeclareMathOperator{\ch}{ch}
\DeclareMathOperator{\Spec}{Spec}
\newcommand{\CC}{\mathbb C}
\newcommand{\ZZ}{\mathbb Z}
\newcommand{\QQ}{\mathbb Q}
\newcommand{\eps}{\varepsilon}

\begin{document}

\title{Schubert calculus and singularity theory}
\author{Vassily Gorbounov and Victor Petrov}

\begin{abstract}
Schubert calculus has been in the intersection of several fast developing areas of mathematics for a long time.
Originally invented as the description of the cohomology of homogeneous spaces it has to be redesigned when applied to other generalized cohomology theories such as the equivariant, the quantum cohomology, K-theory, and cobordism.
All this cohomology theories are different deformations of the ordinary cohomology. In this note we show that there
is in some sense the universal deformation of Schubert calculus which produces the above mentioned by specialization of the appropriate parameters.
We build on the work of Lerche Vafa and Warner. The main conjecture these authors made was that the classical cohomology of a hermitian symmetric homogeneous manifold is a Jacobi ring of an appropriate potential.
We extend this conjecture and provide a simple proof. Namely we show that the cohomology of the hermitian symmetric space is a Jacobi ring of a certain potential and the equivariant and the quantum cohomology and the K-theory is a Jacobi ring of a particular deformation of this potential.
This suggests to study the most general deformations of the Frobenius algebra of cohomology of these manifolds
by considering the versal deformation of the appropriate potential. The structure of the Jacobi ring of such potential is a subject of well developed singularity theory. This gives a potentially new way to look at the classical, the equivariant, the quantum and other flavors of Schubert calculus.
\end{abstract}

\maketitle

\section{Introduction}
The connection between the cohomology of the homogeneous space $G/H$ and the quantum field theory called the coset model $\bf G/H$ was stated in the influential paper by Lerche Vafa and Warner \cite{LVW}. The key conjectural claim was that the ring $\HH^*(G/H)$ is the Jacobi ring of some function, the so called potential. In \cite{LVW} and subsequent publications this conjecture was supported by a number of non-trivial examples. Later Gepner in \cite{G} has calculated the potential for the cohomology of any Grassmann manifold $\Gr(k,n)$ and discovered that the cohomology of the Grassmann manifolds and the fusion rings of the unitary group are Jacobi ring of the same potential. This work was followed by Witten in \cite{W} where it was pointed out that this connection extends naturally to the quantum cohomology of the Grassmann manifold if one considers the deformed Gepner potential. Also in \cite{FLMW} it was argued that a certain deformation of the potential provides the connection with the affine Toda lattice type field theory. This gave a beautiful link between the representation theory of the loop groups, the quantum Schubert calculus, the integrable systems and the singularity theory.

We take on the task of extending the conjecture of Lerche, Vafa and Warner about the cohomology of hermitian symmetric spaces as defined in \cite{BH} and providing simple proof of it.
We show that the classical cohomology of these homogeneous spaces are Jacobi rings of an appropriate Taylor polynomial, the potential, which we calculate explicitly, and the equivariant cohomology with respect to the torus action, the small quantum cohomology and K-theory are the Jacobi rings of a particular deformation of the potential.  Our approach offers a possibly new way to look on the classical as well as the equivariant and quantum formulas from the Schubert calculus like the Littlewood--Richardson rule from the point of view of singularity theory \cite{AGZV} \cite{S}. It is interesting to explore the connection between these two ways of describing the same algebra. In the case of the Grassmann manifold such an attempt is available in the literature \cite{Ch}.
Further using the description of the equivariant cohomology with respect to the torus action of these homogeneous spaces as a Jacobi ring we give the description of the scheme $\Spec\HH^*_T(G/P)$ in the spirit of the work \cite{GMcPh}. It turns out that the equivariant Euler classes of the tangent bundle at a fixed point of the torus action are the values of the Hessian of our potential at the fixed point. Next we give a description of the small quantum cohomology of these spaces as a scheme and make explicit the Vafa--Intriligator formula by calculating the denominators in it. Namely the Gromov--Witten invariants can be calculated as a sum over the vertices of an appropriate polytope with the explicitly calculated weights. The vertices are identified using Kostant's description of cyclic elements.
As an interesting outcome of this result we give a fast algorithm for calculating Littlewood--Richardson coefficients for the product of Schubert classes.

The connection between the representation theory and singularity theory was the subject of the conjecture by Zuber solved in \cite{VGZ}.

In the modern language of Mirror Symmetry such a presentation of cohomology as a Jacobi ring is called the Landau--Ginzburg model of the homogeneous manifold. This places the study of the deformations of the Frobenius algebra of cohomology of these manifolds into the Saito theory of the Frobenius structure of singularity \cite{S}.
We would like to point out that our potential is a Taylor polynomial and is different from the potential obtained out of toric degeneration of the homogeneous space, like in the approach developed in \cite{EH}, \cite{HV}, \cite{Giv}, \cite{Bat}. The modern treatment of these ideas is given for example in \cite{Pr} where such a potential is called the weak Landau Ginzburg model. The connection of our potential to the weak Landau Ginzburg model and to the quantum cohomology is now being investigated
\cite {GMS}.



The result of \cite{HFT} suggests the relation of the quantum cohomology of homogeneous spaces and the twisted $K$ theory via the Verlinde algebra. It would be interesting to investigate it. Also following the approach of \cite{B} and \cite{GM} one may be able to obtain the Chiral de Rham complex of the homogeneous spaces as the ``Jacobi ring'' of the ``chiralisation'' of the potential. We plan to return to this point in future work.

The authors visited the Max Planck Institute for Mathematics in Bonn while working on this paper. Inspiring discussions with V. Golyshev F. Hirzebruch and Yu. I. Manin are gratefully acknowledged.

\section{Summary of Gepner's work}
Let us remind the Gepner result on the cohomology of the Grassmannian.

The cohomology of the Grassmannian $\Gr(k,n)=\UU(n)/\UU(k)\times\UU(n-k)$ have the following presentation:
let $C=1+c_1+\ldots+c_k+\ldots$ is the full class of the tautological bundle and $\bar C=1+\bar c_1+\ldots+\bar c_{n}+..$, the relation between them is $C\bar C=1$.
For our purposes we need the following description of $\HH^*(\Gr(k,n),\ZZ)$:
$$\HH^*(\Gr(k,n),\ZZ) = \ZZ[c_1,\ldots,c_k]/\bar c_{n-k+1}=\ldots=\bar c_n=0$$

The existence of the potential can be seen using the following easy observation: consider the relations in the cohomology of the Grassmannian
\begin{itemize}
\item $c_1+\bar c_1$
\item $c_2+c_1\bar c_1 +\bar c_2$
\item $c_3+c_2\bar c_1 + c_1\bar c_2 + \bar c_3$
\item $c_4+c_3\bar c_1 + c_2\bar c_2 + c_1\bar c_3 + \bar c_4$
\item $c_5+c_4\bar c_1+c_3\bar c_2 + c_2\bar c_3 + c_1\bar c_4 + \bar c_5$
\end{itemize}

\begin{prop} Consider $\bar c_i$ as a function in $c_1,c_2,\ldots$ defined by the above equation $C\bar C=1$. The following identity holds:
$$\frac{\partial \bar c_i}{\partial c_j}=\frac{\partial \bar c_{i+l}}{\partial c_{j+l}}$$
\end{prop}
\begin{proof}
$\bar C$ is a function in $c_1,c_2,\ldots$.
Differentiating the identity $C\bar C=1$ with respect to $c_i$ produces
$$\frac{\partial}{\partial c_i}\bar C=-\frac{\bar C}{C}$$
In other words the function $\bar C$ satisfies the system of linear partial differential equations:
$$\frac{\partial }{\partial c_{i}}\bar C=\frac{\partial }{\partial c_{j}}\bar C$$
Taking the graded components of this identity proves the proposition.
\end{proof}
This implies that the ideal defining the cohomology ring of the Grassmann manifold is a Jacobi ideal.
Figuring out the potential is easy now.
The function $\bar C$ has the form $\frac{1}{C}$. Therefore the generating function for the potentials is
$$V=\log (C)$$
Taking now the graded component of the weight $n+k-1$ and setting $c_i=0$ for $i>k$ we obtain the Landau--Ginzburg potential for $\Gr(k,n)$.
Here is the potential for $\Gr(2,5)$
$$V=-\frac{c_1^6}{6}+c_1^4c_2-\frac{3}{2}c_1^2c_2^2+\frac{c_2^3}{3}$$

It is easy now to obtain the Gepner description of the generating function for the potential:
$$V=\log (1+c_1+c_2+\ldots)=\log \Pi(1+t_i)=\sum \log (1+t_i)=\sum_{i} t_i-\sum_i \frac{t_i^2}{2}+\sum_i\frac{t_i^3}{3}-\ldots$$


\section{The cohomology of the orthogonal Grassmann manifold and the Catalan numbers}

The cohomology of the orthogonal Grassmannian $\OGr(n,2n)=\SO(2n)/\UU(n)$ can be
presented as follows \cite{BoeHiller}. Consider a polynomial ring $\ZZ[x_1,x_2,\ldots,x_{2n-2}]$, $\deg x_i=2i$, and factor out
the set of elements $$P_{i} = x_i^2 + 2\sum_{k=1}
(-1)^k x_{i+k}x_{i-k} + (-1)^j x_{2i}.$$
As usual we set $x_i$ to be zero if the number $i$ is outside of the interval $[1,2n-2]$.
A few of the elements $P_{i}$ are listed below:
\begin{itemize}
\item $x_1^2-x_2$
\item $x_2^2-2x_1x_3+x_4$
\item $x_3^2-2x_2x_4+2x_1x_5-x_6$
\item $x_4^2-2x_3x_5+2x_2x_6-2x_1x_7+x_8$
\item $x_5^2-2x_4x_6+2x_3x_7-2x_2x_8+2x_1x_9-x_{10}$
\end{itemize}
The meaning of the variables $x_i$ is as follows: it is the Chern classes of the tautological bundle divided by 2 \cite{BoeHiller}. Introduce two generating functions $X_+=1+x_2+x_4+\ldots$ and $X_-=x_1+x_3+x_5+\ldots$
Considering the even $x_{2n}$ as functions of the odd $x_{2m+1}$'s the above relations can be written as
$(X_++X_-)(X_+-X_--1)=X_-$
\begin{prop} One has
$$\frac{\partial x_{2(n+k)}}{\partial x_{2(m+k)+1}}=\frac{\partial x_{2n}}{\partial x_{2m+1}}$$
\end{prop}
This implies the existence of the potential for the ideal generated by the $P_i$'s. Let us calculate it.
Solving for $X_+$ one obtains
$$X_+=-\frac{1}{2}+\sqrt{\frac{1}{4}+X^2_-}$$
Since $$\int  \sqrt{1 + u^2} dx = \frac{u}{2}\big (\sqrt{ 1 + u^2} + \ln (\sqrt{1 + u^2}+u)\big )$$
we obtain the generating function for the potentials
$$V=\frac{1}{2}(-X_-+\big (\sqrt{ 1 + X_-^2} + \ln (\sqrt{1 + X_-^2}+X_-)\big )$$
The component of a given degree $2n-1$ of $V$ gives the potential for the cohomology of $\HH^*(\OGr(n,2n))$.

Here is the Landau--Ginsburg potentials for $\OGr(4,8)$, $\OGr(5,10)$, and $\OGr(6,12)$.
$$V_7=x_1x_3^2-x_1^4x_3+\frac{2}{7}x_1^7$$
$$V_9=-2x_1^3x_3^2+2x_1^6x_3+\frac{1}{3}x_3^3-\frac{5}{9}x_1^9$$
$$V_{11}=2x_1^6x_5+6x_1^5x_3^2-2x_1^2x_3^3+x_1x_5^2+x_3^2x_5-5x_1^8x_3-4x_1^3x_3x_5+\frac{14}{11}x_1^{11}$$

\begin{rem} It is interesting to note that the above relations are related to the Catalan numbers. Namely
the expansion of the function
$$g(u)=-\frac{1}{2}+\sqrt{\frac{1}{4}+u^2_-}$$
is the generating function for the Catalan numbers
$$g(u)=\sum (-1)^nC_{n-1}u^{2n}$$
Indeed setting the odd $x_i$ for $i>1$ in the defining relations equal to zero one obtains Segner's Recurrence Formula for the Catalan number.
\end{rem}

\begin{rem} Since $B_n$ and $C_n$ has the same Weyl group, and their torsion index
is $2$, one has
$$
\HH^*(\Sp(n)/\UU(n),\ZZ\Big[\frac{1}{2}\Big])\simeq\HH^*(\OGr(n+1,2n+2),\ZZ\Big[\frac{1}{2}\Big]),
$$
so the result applies to the symmetric hermitian space $\Sp(n)/\UU(n)$ as well.
\end{rem}


\section{The Quadrics}

In this section we consider the case of even-dimensional quadrics
$Q_{2n-2}=\SO(2n)/\SO(2)\times\SO(2n-2)$. Note that for odd-dimensional quadrics
we have
$$
\HH^*(Q_{2n-1},\ZZ\Big[\frac{1}{2}\Big])=\HH^*(\Gr(1,2n),\ZZ\Big[\frac{1}{2}\Big]),
$$
the case already considered above.

It is well-known that $\HH^*(Q_{2n-2},\ZZ)$ has a presentation
$$
\ZZ[h,l]/(h^n-2hl,\,l^2)
$$
when $n$ is odd and
$$
\ZZ[h,l]/(h^n-2hl,\,l^2-h^{n-1}l)
$$
when $n$ is even, where $h$ is the class of a hyperplane section and
$l$ is the class of a maximal totally isotropic subspace.

Consider first the case of odd $n$. The function
$$
V(h,l)=-h^nl+hl^2+\frac{n}{2(2n-1)}h^{2n-1}
$$
satisfies the equations
\begin{align*}
&\partial_l V=-(h^n-2hl);\\
&\partial_h V=l^2+\frac{n}{2}h^{n-2}(h^n-2hl),
\end{align*}
so $V$ is a potential.

Similarly, in the case of even $n$ we set
$$
V(h,l)=-h^nl+hl^2+\frac{n-1}{2(2n-1)}h^{2n-1}.
$$
Then
\begin{align*}
&\partial_l V=-(h^n-2hl);\\
&\partial_h V=l^2-h^{n-1}l+\frac{n-1}{2}h^{n-2}(h^n-2hl),
\end{align*}
and $V$ is a potential.

\section{Cohomology of $E_6/\Spin(10)\times\SO(2)$, $E_7/E_6\times\SO(2)$ and $E_8/E_7\times\SO(2)$}

In \cite{DZ} (see also \cite{CM}) the following presentation was obtained:
$$
\HH^*(E_6/\Spin(10)\times\SO(2),\ZZ)=\ZZ[y_1,y_4]/(r_9,\,r_{12}),
$$
where
\begin{align*}
&r_9=2y_1^9+3y_1y_4^2-6y_1^5y_4;\\
&r_{12}=y_4^3-6y_1^4y_4^2+y_1^{12}.
\end{align*}

Setting
$$
V(y_1,y_4)=2y_1^9y_4+y_1y_4^3-3y_1^5y_4^2-\frac{5}{13}y_1^{13}
$$
we see that
\begin{align*}
&\partial_{y_4}V=r_9;\\
&\partial_{y_1}V=r_{12}-3y_1^3r_9,
\end{align*}
so $V$ is a potential.

Similarly, a presentation for $\HH^*(E_7/E_6\times\SO(2),\ZZ)$ looks as follows:
$$
\ZZ[y_1,y_5,y_9]/(r_{10},r_{14},r_{18}),
$$
where
\begin{align*}
&r_{10}=y_5^2-2y_1y_9;\\
&r_{14}=2y_5y_9-9y_1^4y_5^2+6y_1^9y_5-y_1^{14};\\
&r_{18}=y_9^2+10y_1^3y_5^3-9y_1^8y_5^2+2y_1^{13}y_5.
\end{align*}

Setting
$$
V(y_1,y_5,y_9)=-y_5^2y_9+y_1y_9^2+3y_1^4y_5^3-3y_1^9y_5^2+y_1^{14}y_5-\frac{2}{19}y_1^{19}
$$
gives
\begin{align*}
&\partial_{y_9}V=-r_{10};\\
&\partial_{y_5}V=-r_{14};\\
&\partial_{y_1}V=r_{18}+2y_1^3y_5r_{10}+2y_1^4r_{14},
\end{align*}
so $V$ is a potential.

Consider now $\HH^{*}(E_{8}/E_{7}\times\SO(2))$. The existence of the potential for this ring has been stated in \cite{LVW} as a problem.

From \cite{DZ} we have:

$$\HH^{*}(E_{8}/E_{7}\times\SO(2))= {\ZZ[y_{1},y_{6},y_{10},y_{15}]}/{(r_{15},r_{20},r_{24},r_{30})},$$
where
\begin{align*}
&r_{15} = 2y_{15} - 16y_{1}^5y_{10} - 10y_{1}^3y_{6}^2 + 10y_{1}^9y_{6} - y_{1}^{15}\\
&r_{20}= 3y_{10}^2 + 10y_{1}^2y_{6}^3 + 18y_{1}^4y_{6}y_{10} - 2y_{1}^5y_{15} - 8y_{1}^8y_{6}^2 + 4y_{1}^{10}y_{10}- y_{1}^{14}y_{6}\\
&r_{24}=5y_{6}^4 + 30y_{1}^2y_{6}^2y_{10} + 15y_{1}^4y_{10}^2 - 2y_{1}^9y_{15} - 5y_{1}^{12}y_{6}^2 + y_{1}^{14}y_{10}\\
&r_{30}= y_{15}^2 - 8y_{10}^3 + y_{6}^5 - 2y_{1}^3y_{6}^2y_{15} + 3y_{1}^4y_{6}y_{10}^2 - 8y_{1}^5y_{10}y_{15} +
 6y_{1}^9y_{6}y_{15} - 9y_{1}^{10}y_{10}^2 - \\
&y_{1}^{12}y_{6}^3 - 2y_{1}^{14}y_{6}y_{10} - 3y_{1}^{15}y_{15} +
 8y_{1}^{20}y_{10} + y_{1}^{24}y_{6} - y_{1}^{30}.
\end{align*}

It is clear that the grading on our algebra and the degrees of the relations prevent the existence of a potential whose gradient would give the above relations. But one can check that the following is true
\begin{prop}
The ring
$${{\Bbb Q}[y_{1},y_{6},y_{10},y_{15}]}/{(r_{15},r_{20},r_{24},r_{30})}$$
is a Jacobi ring with respect to the 'modified gradient' $\Big(y_1\frac{\partial}{\partial y_1},\frac{\partial}{\partial y_6},\frac{\partial}{\partial y_{10}},\frac{\partial}{\partial y_{15}}\Big)$.
The potential is
\begin{align*}
&V=y_{15}^2 - y_{15} y_{1}^{15} + 10 y_{15} y_{1}^9 y_{6} - 10 y_{15} y_{1}^3 y_{6}^2\\
&-16 y_{15} y_{1}^5 y_{10} + 8 y_{1}^20 y_{10} - 80 y_{1}^{14} y_{6} y_{10} + 80 y_{1}^8 y_{6}^2 y_{10}\\
&+64 y_{1}^{10} y_{10}^2+-5 y_{1}^{24} y_{6} + 30 y_{1}^{18} y_{6}^2 - 50 y_{1}^{12} y_{6}^3 + 25 y_{1}^6 y_{6}^4+\frac{y_{1}^{30}}{4}.
\end{align*}
\end{prop}

\begin{rem}
One can check a similar statement for all \emph{extra-special} homogeneous
spaces, that is corresponding to a parabolic subgroup whose unipotent radical
has a one-dimensional commutator subgroup.
\end{rem}

\section{Quantum cohomology and Vafa--Intriligator formula}

Let $X=G/H$ be a hermitian symmetric space. The main result of \cite{CM} says
that the quantum cohomology of $X$ are given by the following deformation of the
ordinary cohomology:
$$
\QH^*(X,\QQ)=\QQ[q,I_1,\ldots,I_n]/(R_1,\ldots,R_{n-1},R_n+q).
$$
If $R_j=I_i$, one can drop $I_i$ and $R_j$ from both lists, and in this way
all presentations considered above are obtained. It follows that $\QH^*(X)$ is
also a Jacobi ring with the potential $V^{[q]}=V+qh$, $V$ being the potential of
$\HH^*(X)$.

Let's look at the solution set of the system of equations
\begin{align*}
&R_1=0\\
&\ldots\\
&R_{m-1}=0\\
&R_m=-q.
\end{align*}
This system describes cyclic elements in the Cartan subalgebra of the Lie algebra
of $G$, and by \cite{K} the Weyl group $W_G$ acts on the set of solutions simply
transitively. This means that the quantum deformation appears to be a particular
case of the equivariant deformation for some choice of $(y_1,\ldots,y_n)$.

Let us describe this choice explicitly for classical groups. In the case of $G=\UU(n)$ one
can choose
$$
R_i=\frac{1}{i}\sum_{j=1}^n t_j^i,
$$
and we can take
$$
(y_1,\ldots,y_n)=\Big(\sqrt[n]{-q},\zeta_n\sqrt[n]{-q},\ldots,\zeta_n^{n-1}\sqrt[n]{-q}\Big),
$$
where $\zeta_n$ is $n$-th root of $1$. The Weyl group $W(A_{n-1})=S_n$ acts
by permutations.

In the case of $G=\SO(2n+1)$ or $\Sp(2n)$ we have
$$
R_i=\frac{1}{2i}\sum_{j=1}^n t_j^{2i},
$$
and we can take the same $(y_1,\ldots,y_n)$ as in the case $\UU(n)$.
The Weyl group $W(B_n)=W(C_n)$ acts by permutations and changes
of signs at any positions.

In the case of $G=\SO(2n)$ we have
\begin{align*}
&R_i=\frac{1}{2i}\sum_{j=1}^n t_j^{2i},\,i<n;\\
&R_n=\prod_{j=1}^n t_j,
\end{align*}
and we can take $y_n=0$ and $(y_1,\ldots,y_{n-1})$ as in the case of $\UU(n-1)$.
The Weyl group $|W(D_n)|$ acts by permutations and changes of signs
at even number of positions.

As above, we denote the image of $Q$ under the map $(I_1,\ldots,I_n)$ by $P$;
it consists of $|W_G/W_H|$ points.

Using \cite[Theorem~4.5]{SG} we obtain the following formula for the Gromov--Witten
invariants which is due to Vafa and Intriligator \cite{Va1}, \cite{Va2} and \cite{In}:

$$
\langle f\rangle_g=\sum_{p\in P}\ch_{\dim X}(T_X)^{g-1}(p)f(p),
$$
where $f$ is a polynomial in $I_1,\ldots,I_n$. In particular, we recover a
formula in \cite{Ch} for the case of Grassmannians.






\section{Equivariant cohomology and Bott denominators}

Let $X=G/H$ be any projective homogeneous variety. The rational cohomology of
$X$ has the Borel presentation
$$
\HH^*(X,\QQ)=\QQ[I_1,\ldots,I_n]/(R_1,\ldots,R_n),
$$
where $I_i$ are the fundamental invariants of $W_H$ in the space of characters
of a maximal torus $T$ of $W_G$, and $R_j$ are the fundamental invariants of
$G$ on the same space considered as functions in $I_i$.

The $T$-equivariant cohomology of $X$ considered as an algebra over
$\HH^*(BT)=\ZZ[y_1,\ldots,y_n]$ are produced by the following deformation
$$
\HH^*_T(X,Q)=\QQ[y_1,\ldots,y_n,I_1,\ldots,I_n]/(R_1-R_1(y_1,\ldots,y_n),\ldots,R_n-R_n(y_1,\ldots,y_n))
$$

Specialize $(y_1,\ldots,y_n)$ to any point ``in general position''
(more precisely, lying in the interior of some Weyl chamber). Then the common
zeroes of all relations considered as functions in weights $t_1,\ldots,t_n$
of $T$ form $W_G$-orbit of $(y_1,\ldots,y_n)$ that will be denoted by $Q$. Note
that $I_i$ take constant values on $W_H$-orbits in $Q$, therefore
the image $P$ of $Q$ under $(I_1,\ldots,I_n)$ consists of $W_G/W_H$ elements.

By \cite[Ch.~IV,\S~5,n.~4]{Bu} the Jacobian $\det(\frac{\partial R_i}{\partial t_j})$
coincides with the product of positive roots of $G$ (considered as linear functions in
$t_i$) up to a constant. Similarly, $\det(\frac{\partial I_i}{\partial t_j})$ is
the product of positive roots of $H$ up to a constant. It follows that
$$
\det\Big(\frac{\partial R_i}{\partial I_j}\Big)=c\cdot\ch_{\dim X}(T_X),
$$
the top Chern class of the tangent bundle of $X$, up to a constant $c$.

The general theory of Artin rings as in \cite{VGZ} implies that the intersection
form on $\HH^*_T(X)$ looks as follows:
\begin{equation}\label{*}\tag{*}
\langle f\rangle=C\cdot\sum_{p\in P}\frac{f(p)}{\ch_{\dim X}(T_X)(p)},
\end{equation}
where $f$ is a polynomial in $I_i$ and $C$ is some constant. Substituting
$f=\ch_{\dim X}(T_X)$ we get the Euler characteristic of $X$ on the left-hand
side and $C\cdot|W_G/W_H|$ on the right-hand side, so actually $C=1$.

\begin{rem}
A similar formula was obtained by Akyildiz and Carell in \cite{AC}, and for
extraordinary cohomology by Bressler and Evens in \cite[Theorem~1.8]{BE}.
There is a more general formula for the Gysin
map $\HH^*(G/H)\to\HH^*(G/K)$ involving summation over $W_K/W_H$; details are
to appear in a joint paper of the second author and Baptiste Calm\'es.
\end{rem}

For example, consider the case of Grassmannian $X=\Gr(k,n)$. If $f$ is a
polynomial in $c_1,\ldots,c_k$ considered as a symmetric function in
$t_1,\ldots,t_k$, we have
$$
\langle f\rangle=\sum_{\{i_1,\ldots,i_k\}\in\Lambda^k\{1,\ldots,n\}}
\frac{f(y_{i_1},\ldots,y_{i_k})}{\prod_{i\in\{i_1,\ldots,i_k\},
j\in\{1,\ldots,n\}\setminus\{i_1,\ldots,i_k\}}(y_i-y_j)}.
$$
In particular, this expression is always a polynomial in $y_i$, which is
constant if $\deg f=k(n-k)$ and $0$ if $\deg f<k(n-k)$.

\begin{rem}
The following observation is due to Fedor Petrov. Applying the Combinatorial
Nullstellensatz \cite[Theorem 4]{KP} to $g=\frac{1}{k!}f\cdot\prod_{1\le i<j\le k}(t_i-t_j)^2$
and $|A_i|=\{y_1,\ldots,y_n\}$, we see that $\langle f\rangle$ is equal up to
sign to the coefficient of $g$ at the monomial $t_1^{n-1}\ldots t_k^{n-1}$,
if $\deg f\le k(n-k)$.
\end{rem}

Another example is the maximal orthogonal Grassmannian $\OGr(n,2n)$. If $f$
is a polynomial in $x_1,\ldots,x_n$ considered as a symmetric function in
$t_1,\ldots,t_n$, we have
$$
\langle f\rangle=\sum_{\eps_i=\pm 1,\prod_i\eps_i=1}
\frac{f(\eps_1y_1,\ldots,\eps_ny_n)}{\prod_{1\le i<j\le n}(\eps_iy_i+\eps_jy_j)}.
$$

\begin{rem} The construction of the previous section allows to describe the equivariant cohomology of the hermitian homogeneous spaces as an affine scheme in the spirit of work \cite{GMcPh}.

We consider the case of the Grassmann manifolds as an illustration.
The compact maximal torus $T=\SO(2)^n$ inside $\UU(n)$ acts naturally on $\Gr(k,n)$.
The equivariant cohomology ring $\HH^*_T(\Gr(k,n))$ is an algebra over the
cohomology of the classifying space of $T$. The latter is isomorphic to
$\ZZ[y_1,\ldots,y_n]$, $\deg y_i=2$, where $y_i$ are the first Chern classes of the
appropriate line bundles over $T$. The relations in the cohomology ring come
from the relation $C\bar C=\prod_{i=1}^n(1+y_i)$. First few relations look like this:
\begin{itemize}
\item $c_1+\bar c_1-e_1({\bf y})$
\item $c_2+c_1\bar c_1 +\bar c_2-e_2({\bf y})$
\item $c_3+c_2\bar c_1 + c_1\bar c_2 + \bar c_3-e_3({\bf y})$
\item $c_4+c_3\bar c_1 + c_2\bar c_2 + c_1\bar c_3 + \bar c_4-e_4({\bf y})$
\item $c_5+c_4\bar c_1+c_3\bar c_2 + c_2\bar c_3 + c_1\bar c_4 + \bar c_5-e_5({\bf y})$
\item $\ldots$
\end{itemize}
Here $e_j({\bf y})$ is the $j$-th elementary symmetric functions in $y_i$.

This allows to view $\bar c_i$ as functions in $c_j$ and $y_j$.
Then
$$\HH^*_T(\Gr(k,n),\ZZ) = \ZZ[c_1,\ldots,c_k]/\bar c_{n-k+1}=\ldots=\bar c_n=0$$
Describing  $\Spec\HH^*_T(Gr(k,n))$ amounts to solving the equations $\bar c_{n-k+1}=\ldots=\bar c_n=0$ in $c_1,\ldots,c_k$ with respect to
$y_1,\ldots,y_n$.
This can be done by a trick similar to the one in the section above.
Take any $k$-element subset of the set $\{y_1,\ldots,y_n\}$, say $\{y_{i_1},\ldots,y_{i_k}\}$. Then the appropriate solution is given
as follows:
\begin{align*}
&c_1=e_1(y_{i_1},\ldots,y_{i_k})\\
&c_2=e_2(y_{i_1},\ldots,y_{i_k})\\
&\ldots\\
&c_k=e_k(y_{i_1},\ldots,y_{i_k})
\end{align*}
\end{rem}

\section{Algorithm for computing Littlewood--Richardson coefficients}

Let $X$ be as in the previous section. There is an important additive basis
of $\HH^*(X)$ consisting of the classes of Schubert varieties $Z_w$ parametrized
by $w\in W_G/W_H$. The coefficients of the multiplication table in this basis
are known as \emph{(generalized) Littlewood--Richardson coefficients}. Note that
the coefficient of the product $[Z_u]$ and $[Z_v]$ at $[Z_w]$ equals
$$
\langle [Z_u][Z_v][Z_{w^\vee}]\rangle,
$$
where $Z_{w^\vee}$ is the dual Schubert variety to $Z_w$. So the
formula~\eqref{*} gives rise to a fast computational tool once one knows how to
express $Z_u$ in terms of multiplicative generators $I_i$. In the classical
case of Grassmannians formulas are given in \cite{FK}; in general we proceed
using the divided difference operators invented by Demazure \cite{D}.

Namely, it is not hard to see that for any point $x$ in general position and
any $W_H$-invariant function $f$ in weights of degree $k$, the element in
$\HH^*(X)$ represented by $f$ has a form
$$
\sum_{l(w)=k}\sum_{u\le w}C_{w,u,x}f(ux)[Z_w],
$$
where $\le$ stands for the Bruhat order and $C_{w,u,v}$ are constants
not depending on $f$. The coefficients $C_{w,u,x}$ can be computed recursively
based on the identities
\begin{align*}
&C_{1,1,x}=1;\\
&C_{s_\alpha w',u,x}=\frac{C_{w',u,x}-C_{w',s_\alpha u,s_\alpha x}}{\alpha(x)},
\end{align*}
where $\alpha$ is any simple root and $s_\alpha$ is the reflection corresponding
to $\alpha$.


\begin{thebibliography}{XXX}
\bibitem{AC} E. Akyildiz, J.B. Carrell, {\em An algebraic formula for the Gysin homomorphism
from $G/B$ to $G/P$}, Illinois J. Math., vol. 31, no. 2 (1987), 312--320.
\bibitem{AGZV}~ V.I. Arnold, S.M. Gusein-Zade, A.N. Varchenko,{\em Singularities of Differentiable Maps: Vol. 1,2} Monographs in Mathematics, Birkhauser, Boston-Basel-Berlin, 1985, 1988.
\bibitem{Bat} V. Batyrev, I. Ciocan-Fontanine, B. Kim, Duco van Straten,{\em Conifold transitions and mirror symmetry for Calabi--Yau Complete Intersections in Grassmannians}
Nuclear Phys. B, vol. 514 (1998), no. 3, 640--666.
\bibitem{BFK}A. Bertram, I. Ciocan-Fontanine, B. Kim, {\em Two proofs of a conjecture of Hori and Vafa},
Duke Math. J., vol. 126, no. 1 (2005), 101--136.
\bibitem{BoeHiller} B.Boe, H.Hiller, {\em Pieri Formula for $SO_{2n+1}/U_n$ and $Sp_n/U_n$}, Adv. in Math., vol. 62 (1986), 49--67.
\bibitem{BH} A. Borel, F. Hirzebruch, {\em Characteristic classes and homogeneous spaces, I},
Amer. J. of Math., vol. 80, no. 2 (1958), 458-538
\bibitem{B} L.Borisov, {\em Vertex algebras and mirror symmetry}, Comm. Math. Phys.
vol. 215, no. 3 (2001), 517--557.
\bibitem{Bu} N.~Bourbaki, {\em Groupes et alg\`ebres de Lie. Chapitres 4, 5 et 6}, Masson, Paris, 1981.
\bibitem{BE} P.~Bressler, S.~Evens.
{\em The Schubert calculus, braid relations, and generalized cohomology},
Trans. Amer. Math. Soc., vol. 317, no. 2 (1990), 799--811.
\bibitem{Ch} N. Chair, {\em Intersection numbers on Grassmannians, and on the space of holomorphic maps from $\CC P^1$ into $G_r(\CC^n)$},  J. Geom. Phys., vol.  38  (2001),  no. 2, 170--182.
\bibitem{CM} P.E.~Chaput, L.~Manivel, {\em Quantum cohomology of minuscule
homogeneous spaces}, Trans. Groups, vol.~13, no.~1 (2008), 47--89.
\bibitem{DZ} H.~Duan, Xu.~Zhao, {\em The Chow rings of generalized Grassmannians}, Found. Comput. Math., vol. 10 (2010), 245--274.
\bibitem{D} M.~Demazure, {\em D\'esingularisation des vari\'et\'es de Schubert
g\'en\'eralis\'ees}, Ann. Sci. \'Ecole Norm. Sup., vol.~7 (1974), 53--88.
\bibitem{EH} T. Eguchi, K. Hori, C.-S. Xiong, {\em Gravitational quantum cohomology}, Internat. J. Modern Phys. A., vol. 12 (1997), 1743--1782.
\bibitem{HV}  K. Hori, C. Vafa, {\em Mirror Symmetry}, 	arXiv:hep-th/0002222v3.
\bibitem{FLMW} P. Fendley, W. Lerche, S.D. Mathur, N.P. Warner, {\em $N=2$ supersymmetric integrable models from affine Toda theories},  Nuclear Phys. B, vol.  348  (1991),  no. 1, 66--88.
\bibitem{FK} S.~Fomin, A.N~Kirillov, {\em The Yang--Baxter equation, symmetric functions, and Schubert polynomials}, Discrete Math., vol. 153 (1996), 123--143.
\bibitem{G} D. Gepner, {\em Fusion rings and geometry}, Commun. Math. Phys., vol. 141 (1991), 381--411.
\bibitem{Giv}  A. Givental, B. Kim. {\em Quantum cohomology of flag manifolds and Toda lattices}, Commun. Math. Phys., vol. 168 (1995), 609--641.
\bibitem{GM} V. Gorbounov, F. Malikov, {\em Vertex algebras and the
Landau--Ginzburg/Calabi--Yau correspondence}
\bibitem{GMS} V.Gorbounov, Yu.I. Manin, M.Smirnov, {\em work in progress}.
\bibitem{GMcPh} M. Goresky and R. MacPherson, {\em On the Spectrum of the Equivariant Cohomology Ring}, Canad. J. Math. 62 (2010), 262--283.
\bibitem{VGZ} S.M. Gusein-Zade, A. Varchenko, {\em Verlinde algebras and the intersection form on vanishing
cycles}, Sel. math., New ser. 3, vol. 79 (1997), 79--97.
\bibitem{HFT} D. Freed, M. Hopkins, C. Teleman, {\em Twisted equivariant K-theory with complex coefficients}. J. Topol., vol. 1, no. 1 (2008), 16--44.
\bibitem{In} K. Intriligator, {\em Fusion Residues}, Mod. Phys. Lett. A, vol. 38 (1991), 3543--3556.
\bibitem{KP}  R.N. Karasev, F.V. Petrov, {\em Partitions of nonzero elements of a finite field into pairs}, arXiv:1005.1177v2.
\bibitem{KT} A. Kresch, H. Tamvakis, {\em Quantum cohomology of orthogonal Grassmannians}, Compositio Mathematica, vol. 140, no. 2 (2004), 482--500.
\bibitem{K} B. Kostant, {\em The principal three-dimensional subgroup and the Betti numbers of a complex simple Lie group}, Amer. J. Math., vol. 81 (1959), 973--1032.
\bibitem{LVW} W. Lerche, C. Vafa, N.P. Warner, {\em Chiral rings in $N=2$ superconformal theories}, Nuclear Phys. B, vol.  324  (1989),  no. 2, 427--474.
\bibitem{Manin} Yu.I.Manin, {\em Frobenius manifolds, quantum cohomology, and moduli spaces}, American Mathematical Society Colloquium Publications, vol. 47.
\bibitem{Pr} V. Przyjalkowski, {\em Weak Landau--Ginzburg models for smooth Fano threefolds}, arXiv:0902.4668v2.
\bibitem{S} K. Saito, {\em Primitive forms for a universal unfolding of a function with an
isolated critical point}, Journ. Fac. Sci. Univ. Tokyo, Sec. IA, vol. 28 (1981), 775--792.
\bibitem{SG} B. Siebert, G. Tian, {\em On Quantum cohomology rings of Fano manifolds and a formula of Vafa and Intriligator}. Asian J. Math, vol. 1, no. 4 (1997), 679--695.
\bibitem{Va1} C. Vafa, {\em Topological Landau--Ginzburg models}, Mod. Phys. Lett. A, vol. 6 (1991) 337-346.
\bibitem{Va2} C. Vafa, {\em Topological Mirrors and Quantum Rings}, in: Essays on Mirror Manifolds (ed.
S.-T. Yau), International Press 1992, 96--119.
\bibitem{W} E. Witten, {\em The Verlinde algebra and the cohomology of the Grassmannian}, In Geometry, topology,
\& physics, Conf. Proc. Lecture Notes Geom. Topology, IV, 357--422. Int. Press, Cambridge,
MA, (1995).





\end{thebibliography}
\end{document}